\documentclass{article}

\usepackage{amsmath}
\usepackage{amssymb}
\usepackage{amsthm}
\usepackage{enumerate}
\usepackage{cases}
\usepackage{stmaryrd}
\usepackage{color}
\usepackage[all]{xy}
\usepackage{hyperref}

\usepackage[top=1in, bottom=1in, left=1in, right=1in]{geometry}

\usepackage{fancyhdr}
\numberwithin{equation}{section}  

\newtheorem{lemma}{Lemma}[section]
\newtheorem{theorem}[lemma]{Theorem}

\newtheorem{corollary}[lemma]{Corollary}

\theoremstyle{definition}
\newtheorem{definition}[lemma]{Definition}

\theoremstyle{remark}
\newtheorem{example}[lemma]{Example}

\newcommand{\ii}{\sqrt{-1}}   

\newcommand{\ba}{\bar{\alpha}}   
\newcommand{\bb}{\bar{\beta}}

\newcommand{\bm}{\bar{\mu}}

\newcommand{\la}{\langle}
\newcommand{\ra}{\rangle}

\newcommand{\Rmnum}[1]{\uppercase\expandafter{\romannumeral #1}}

\title{Pseudo-Harmonic Maps From Pseudo-Hermitian Manifolds to Riemannian Manifolds}

\author{Yibin Ren\footnote{Supported by NSFC Tianyuan fund for Mathematics Grant No. 11626217.} \and Guilin Yang}

\date{}

\begin{document}

\maketitle

\begin{abstract}
    In this paper, we discuss the heat flow of a pseudo-harmonic map from a closed pseudo-Hermitian manifold to a Riemannian manifold with non-positive sectional curvature, and prove the existence of the pseudo-harmonic map which is a generalization of Eells-Sampson's existence theorem. We also discuss the uniqueness of the pseudo-harmonic representative of its homotopy class which is a generalization of Hartman theorem, provided that the target manifold has negative sectional curvature.
\end{abstract}

\section{Introduction}

Let $(M^{2m +1}, HM, J_b, \theta)$ be a closed pseudo-Hermitian manifold with horizontal bundle $HM$, almost complex structure $J_b$, pseudo-Hermitian structure $\theta$ and real dimension $2m +1$. There is naturally a sub-Riemannian structure which contains the horizontal bundle and the sub-Riemannian metric $G_\theta$ (See Section \ref{s2} for details). Suppose $(N, h)$ is a Riemannian manifold and $f : M \to N$ is a smooth map. The horizontal energy $E_H$ is defined by 
\begin{align*}
	E_H (f) = \frac{1}{2} \int_M \langle G_\theta , f^* h \rangle \theta \wedge (d \theta)^m = \frac{1}{2} \int_M | d f \circ \pi_H |^2 \theta \wedge (d \theta)^m
\end{align*}
where $\pi_H : TM \to HM$ is the horizontal projection. A smooth map $f : M \to N$ is called pseudo-harmonic by E. Barletta, S. Dragomir and H. Urakawa \cite{barletta2001pseudoharmonic} if it is a critical point of $E_H$ whose local Euler-Lagrange equation is
\begin{align*}
	\Delta_H f^i + \Gamma^i_{jk} \langle \nabla_H f^j, \nabla_H f^k \rangle =0,
\end{align*}
where $\Delta_H = div (\nabla_H)$ is the sub-Laplacian and $\nabla_H f^j = \pi_H \nabla f^j$ is the horizontal gradient of $f^j$. Clearly, it is also called subelliptic harmonic map in sub-Riemannian terminology (cf. \cite{jost1998subelliptic,zhou2013heat}). 

Due to the variational structure of horizontal energy $E_H$, a natural question is the existence of pseudo-harmonic maps or subelliptic maps. Under a convexity condition of target manifold, Jost and Xu \cite{jost1998subelliptic} obtained the existence of the Dirichlet problem of subelliptic harmonic maps. Through the heat flow method, Zhou \cite{zhou2013heat} deduced Eells-Sampson's existence theorem for subelliptic harmonic maps under the $\Gamma$-tensor vanishing condition of the domain. It is notable that the $\Gamma$-tensor of a pseudo-Hermitian manifold won't vanish since $HM$ satisfies the strong bracket generating hypothesis. For pseudo-harmonic case, S. C. Chang and T. H. Chang \cite{chang2013existence} used the heat flow method again and proved Eells-Sampson's existence theorem under the condition $[\Delta_H, T] =0 $ on the domain where $\Delta_H$ is the sub-Laplacian and $T$ is the Reeb vector field. It is also remarkable that Jost and Yang \cite{jost2005heat} discussed the heat flows of horizontal harmonic maps from Riemannian manifolds to Carnot-Carath\'eodory space.

In this paper, we will improve the method of S. C. Chang and T. H. Chang \cite{chang2013existence} to obtain the complete Eells-Sampson's existence theorem for pseudo-harmonic maps without the commutation condition. We will also obtain the Hartman's type theorem, that is any homotopy class from a pseudo-Hermitian manifold to a Riemannian manifold with negative sectional curvature has a unique pseudo-harmonic map. 
The paper is arranged as follows. Section \ref{s2} briefly recalls the basic notions in pseudo-Hermitian geometry and pseudo-harmonic maps; Section \ref{s3} introduces the $L^p$ theory of hypoelliptic operators to derive the regularity of pseudo-harmonic maps and heat flows; Section \ref{s4} deduces the short-time existence of a pseudo-harmonic heat flow by the heat kernel of the subelliptic parabolic operator; In Section \ref{s1}, we control the total energy by the horizontal energy near the initial time which guarantees the long-time existence of a pseudo-harmonic heat flow and its convergence to a pseudo-harmonic map; In Section \ref{s5}, the Hartman type theorem for pseudo-harmonic maps will be proved.

\section{Pseudo-Hermitian Manifolds and Pseudo-Harmonic Maps} \label{s2}
In this section, we present some basic notions and of pseudo-Hermitian geometry and pseudo-harmonic maps. For details, the readers may refer to \cite{dragomir2006differential,Ren201447,webster1978pseudo}. Recall that a smooth manifold $M$ of real dimension ($2n+1$) is said to be a CR manifold if
there exists a smooth rank $n$ complex subbundle $T_{1,0} M \subset TM \otimes \mathbb{C}$ such that
\begin{gather}
T_{1,0} M \cap T_{0,1} M =0 \\
[\Gamma (T_{1,0} M), \Gamma (T_{1,0} M)] \subset \Gamma (T_{1,0} M) \label{a-integrable}
\end{gather}
where $T_{0,1} M = \overline{T_{1,0} M}$ is the complex conjugate of $T_{1,0} M$.
Equivalently, the CR structure may also be described by the real subbundle $HM = Re \: \{ T_{1,0}M \oplus T_{0,1}M \}$ of $TM$ which carries a almost complex structure $J_b : HM \rightarrow HM$ defined by $J_b (X+\overline{X})= i (X-\overline{X})$ for any $X \in T_{1,0} M$.
Since $HM$ is naturally oriented by the almost complex structure $J_b$, then $M$ is orientable if and only if
there exists a global nowhere vanishing 1-form $\theta$ such that $ HM = Ker (\theta) $.
Any such section $\theta$ is referred to as a pseudo-Hermitian structure on $M$. The space of all pseudo-Hermitian structure is 1-dimensional The Levi form $L_\theta $ of a given pseudo-Hermitian structure is defined by
$$L_\theta (X, Y ) = d \theta (X, J_b Y)  $$
for any $X, Y \in HM$. 
An orientable CR manifold $(M, HM, J_b)$ is called strictly pseudo-convex if $L_\theta$ is positive definite for some $\theta$.

When $(M, HM, J_b)$ is strictly pseudo-convex, there exists a pseudo-Hermitian structure $\theta$ such that $L_\theta$ is positive. The quadruple $( M, HM, J_b, \theta )$ is called a pseudo-Hermitian manifold. This paper is discussed in the pseudo-Hermitian manifolds.

For a pseudo-Hermitian manifold $(M, HM, J_b, \theta)$, there exists a unique nowhere zero vector field $T$, called the Reeb vector field, transverse to $HM$ satisfying
$T \lrcorner \: \theta =1, \ T \lrcorner \: d \theta =0$. 
There is a decomposition of the tangent bundle $TM$: 
\begin{align}
TM = HM \oplus \mathbb{R} T \label{a-webstermetric}
\end{align}
which induces the projection $\pi_H : TM \to HM$. Set $G_\theta = \pi_H^* L_\theta$. Since $L_\theta$ is a metric on $HM$, it is natural to define a Riemannian metric
\begin{align}
g_\theta = G_\theta + \theta \otimes \theta
\end{align}
which makes $HM$ and $\mathbb{R} T$ orthogonal. Such metric $g_\theta$ is called Webster metric. In this paper, we denote it by $\la \cdot , \cdot \ra$.
In the terminology of foliation geometry, $\mathbb{R} T$ provides a one-dimensional Reeb foliation and $HM$ is its horizontal distribution.
By requiring that $J_b T=0$, the complex structure $J_b$ can be extended to an endomorphism of $TM$.
The integrable condition \eqref{a-integrable} guarantees that $g_\theta$ is $J_b$-invariant.
Clearly $\theta \wedge (d \theta)^n$ differs a constant with the volume form of $g_\theta$.
Henceforth we will regard it as the volume form and always omit it or denote it by $d V$ for simplicity.

On a pseudo-Hermitian manifold, there exists a canonical connection $\nabla^M$ preserving the horizontal bundle, the CR structure and the Webster metric. Moreover, its torsion satisfies
\begin{gather*}
T_{\nabla^M} (X, Y)= 2 d \theta (X, Y) T  \mbox{ and }T_{\nabla^M} (T, J_b X) + J_b T_{\nabla^M} (T, X) =0.
\end{gather*}


The pseudo-Hermitian torsion, denoted by $A$, is a symmetric tensor defined by $A = g_\theta ( T_{\nabla^M} (T,X), Y)$ for any $X, Y \in TM$ (cf. \cite{dragomir2006differential}).
A pseudo-Hermitian manifold is called Sasakian if $A \equiv 0$.

Let $(M, HM, J_b, \theta)$ be a pseudo-Hermitian manifold of dimension $2n+1$. Let $\{ \eta_1, \dots, \eta_n \}$ be a local orthonormal frame of $T_{1,0} M$ defined on an open set $U \subset M$ , and $\{ \theta^1, \dots \theta^n \}$ its dual coframe. 
Then the structure equations are given by
\begin{numcases}{ }
d \theta = 2 \ii \theta^\alpha \wedge \theta^{\ba} , \nonumber \\
d \theta^\alpha = \theta^\beta \wedge \theta^\alpha_\beta + A_{\bar{\alpha} \bar{\beta}} \theta \wedge \theta^\beta , \nonumber  \\
\theta^\alpha_\beta + \theta^{\bb}_{\ba}= 0 , \nonumber \\
d \theta^\alpha_\beta = \theta^\gamma_\beta \wedge \theta^\alpha_\gamma + \Pi^\alpha_\beta \nonumber
\end{numcases}
where $\theta^\alpha_\beta$'s are the Tanaka-Webster connection 1-forms with respect to $\{ \eta_\alpha \}$. S. M. Webster \cite{webster1978pseudo} showed that
\begin{align*}
\Pi_\beta^\alpha= 2 \ii (\theta^\alpha \wedge \tau_\beta + \theta_\beta \wedge \tau^\alpha) + R^\alpha_{\beta \lambda \bm} \theta^\lambda \wedge \theta^{\bm} + A^\alpha_{\bm, \beta} \theta \wedge \theta^{\bm} - A_{\mu \beta ,}^{\quad \; \alpha} \theta \wedge \theta^\mu
\end{align*}
where $R^\alpha_{\beta \lambda \bm}$ is called the Webster curvature.
He also derived the first Bianchi identity, i.e.  $R_{\ba \beta \lambda \bm} = R_{\ba \lambda \beta \bm}$. So the pseudo-Hermitian Ricci curvature can be defined by $R_{\lambda \bm}= R_{\ba \alpha \lambda \bm}$ and then the pseudo-Hermitian scalar curvature is $ R = R_{\alpha \ba}=R_{\bb \beta \alpha \ba}$.  For more discussion of curvatures, one can refer to \cite{dragomir2006differential}.

Assume that $(N,h)$ is a Riemannian manifold. Let $\{ \sigma^i \}$ be an orthonormal frame of $T^*N$ and $\{ \xi_i \}$ its dual frame of $TN$. Denote by $\nabla^N$ the Levi-Civita connection of $(N,h)$. Suppose that $f: M \rightarrow N$ is a smooth map. The pullback connection $\nabla$ on the pullback bundle $f^*(TN)$ and the Tanaka-Webster connection induce a connection on $TM \otimes f^*(TN)$, also denoted by $\nabla$. Let $f^i_{AB}$ be the components of $\nabla df$ under the frame $\{ \theta^A \} = \{ \theta^0=\theta, \theta^\alpha, \theta^{\ba} \}$ and $\{ \xi_i \}$.

\begin{definition}
A smooth map $f: M \rightarrow N$ is called pseudo-harmonic if the tensor field $$\tau (f) = trace_{G_\theta} \nabla d f |_{HM \times HM} $$ vanishes.
\end{definition}

Actually, pseudo-harmonic maps are the critical points of the horizontal energy (cf. \cite{dragomir2006differential})
\begin{align}
E_H (f) =  \int_M e_H (f) \theta \wedge (d \theta)^m
\end{align}
where $e_H (f) =\frac{1}{2} |d_H f|^2$ is the horizontal energy density.

Suppose that $N$ is compact. By the Nash embedding theorem, $(N,h)$ can be isometrically embedded in $\mathbb{R}^{K}$ for some positive integer $K$. The embedding map is denoted by $\iota$. The tubular neighborhood theorem guarantees that there exists a neighborhood $B(N)$ of $\iota (N)$ which is diffeomorphic to a neighborhood of the zero section in the normal bundle, and a projection map $P : B(N) \rightarrow \iota (N)$ that is a submersion. Moreover, the projection map is given by mapping any point in $B (N)$ to its closest point in $N$. Clearly $\iota = P \circ \iota$ and for any $x \in \iota (N)$, $d P_x : T_x \mathbb{R}^K \rightarrow T_x \mathbb{R}^K$ annihilates all vertical vector fields. Let $f :  M \rightarrow N$ be a smooth map and $u = \iota \circ f$. Now we deduce the expression of $\tau  (f)$ under the coordinate components of $u$. The composition law (cf. Proposition 2.20 in \cite{eells1983selected}) implies that
\begin{align*}
\tau (u) =& d \iota (\tau  (f) ) + trace_{G_\theta} (\nabla d \iota) (d_H f, d_H f), \\
\nabla d \iota =& d P ( \nabla d \iota) + \nabla d P (d \iota, d \iota)  ,
\end{align*}
where $\nabla d \iota $ is the second fundamental form of $N$ and $d_H f= df |_{HM}$, then $\nabla d \iota = \nabla d P (d \iota, d \iota)$ and thus
\begin{align*}
d \iota (\tau  (f) )  = \tau (u) - trace_{G_\theta} (\nabla d P) (d_H u, d_H u).
\end{align*}
Let $\{ x^a \}$ be the natural coordinates of $\mathbb{R}^K$ and $u^a = x^a \circ u$, $P^a = x^a \circ P$. Then
\begin{align}
d \iota (\tau  (f) )  = \left( \Delta_H u^a - P^a_{bc} (u) \la \nabla_H u^b , \nabla_H u^c \ra \right) \frac{\partial }{\partial x^a}. \label{a4}
\end{align}
where $\Delta_H$ is the sub-Laplacian and $P^a_{bc} = \frac{\partial^2 P^a }{\partial x^b \partial x^c}$.

\begin{lemma}
Assume that $f : M \rightarrow N$. Then $f$ is pseudo-harmonic if and only if
\begin{align}
\Delta_H u^a - P^a_{bc} (u) \la \nabla_H u^b , \nabla_H u^c \ra = 0. \label{a8}
\end{align}
\end{lemma}

As the harmonic map, to obtain the existence of pseudo-harmonic map, one way is to solve the pseudo-harmonic flow
\begin{align}
\frac{\partial f}{\partial t} = \tau  (f) \label{a1}
\end{align}
where $f : [0, +\infty) \times M \rightarrow N $. We define the map $ \rho : B(N) \rightarrow \mathbb{R}^K $ by
\begin{align*}
\rho(p) = p- P(p).
\end{align*}
Clearly, $\rho (p) $ is normal to $N$ and $\rho (p) =0$ if and only if $p \in N$. The next lemma establishes the fact that in order to solve \eqref{a1}, it suffices to solve the system
\begin{align}
\frac{\partial u^a}{\partial t} = \Delta_H u^a - P^a_{bc} (u) \la \nabla_H u^b , \nabla_H u^c \ra. \label{a2}
\end{align}

\begin{lemma}
Let
$$
u = (u^1 , \dots, u^K)\in C^\infty ( M \times (0, T_0) , B(N) ) \cap C^0 (M \times [0, T_0), B(N))
$$
for some $T_0 \in (0, \infty] $ be a solution of \eqref{a2} with the initial condition $\phi = (\phi^1, \dots, \phi^K) \in C^\infty ( M , B(N) )$, then the quantity
\begin{align*}
\int_M | \rho (u(x ,t)) |^2
\end{align*}
is a nonincreasing function of $t$. In particular, if $\phi (M ) \subset N$, then $f ( x, t) \in N$ for all $(x, t) \in M \times (0, T_0)$.
\end{lemma}

\begin{proof}
Set $P^a_b = \frac{\partial P^a}{\partial x^b}$. Then $\rho^a_b = \frac{\partial \rho^a}{\partial x^b} = \delta^a_b - P^a_b$.
The composition law implies that
\begin{align*}
\Delta_H (P ( u) )^a = P^a_b (u) \Delta_H u^b + P^a_{bc} (u) \la \nabla_H u^b , \nabla_H u^c \ra.
\end{align*}
Thus we have
\begin{align*}
\Delta_H (\rho (u))^a = & \triangle_{H} u^a - \Delta_H (P ( u ) )^a \\
= & \rho^a_b (u) \Delta_H u^b - P^a_{bc} (u) \la \nabla_H u^b , \nabla_H u^c \ra \\
= & \rho^a_b (u) \partial_t  u^b + P^a_b (u)  ( \partial_t - \Delta_H ) u^b
\end{align*}
where the last equality is due the equation \eqref{a2} and $\partial_t = \frac{\partial}{\partial t}$. Since
$$
\big( P^a_b (u)  ( \partial_t - \Delta_H ) u^b \big) \frac{\partial }{\partial x^a} = d P \bigg( ( \partial_t - \Delta_H )  u^a \frac{\partial }{\partial x^a} \bigg)
$$
is tangent to $N$ and $\rho (u)$ is normal vector field, we find
\begin{align*}
\sum_a (\rho (u))^a \Delta_H (\rho (u))^a = \sum_{a, b} (\rho (u))^a \rho^a_b (u) \partial_t  u^b .
\end{align*}
Hence
\begin{align*}
\frac{d}{dt} \int_M | \rho (u(x ,t)) |^2 =& 2 \int_M \sum_{a, b} (\rho (u))^a \rho^a_b (u) \frac{\partial u^b}{\partial t}  =  2 \int_M \sum_a (\rho (u))^a \Delta_H (\rho (u))^a \\
= & - 2 \int_M \sum_a | \nabla_H (\rho (u))^a |^2 \leq  0
\end{align*}
which yields that $ \int_M | \rho (u(x ,t)) |^2$ is monotonic decreasing.
\end{proof}

By using \eqref{a4}, we have the following theorem.

\begin{theorem} \label{at3}
Suppose that $(M, HM, J_b, \theta)$ is a closed pseudo-Hermitian manifold and $(N, h)$ is a compact Riemannian manifold which is isometrically embedded by $\iota$ in $(\mathbb{R}^K, g_{can})$. Let $B(N)$ be the tubular neighborhood of $\iota (N)$ and $P: B(N) \rightarrow \iota (N) $ the closest point projection map. Let $\phi :  M \rightarrow N $ be a smooth map from $M$ into $N \subset \mathbb{R}^K$ given by $\phi = (\phi^1, \dots, \phi^K)$ in the coordinates $\{ x^a \}$. Assume that
$$
u = (u^1, \dots, u^K) \in C^\infty (M \times  (0, T_0), B(N))  \cap C^0 (M \times [0, T_0),  B(N))
$$
for some $T_0 \in (0, \infty]$ is the solution of the subelliptic parabolic system
\begin{align}
\frac{\partial u^a}{\partial t} = \Delta_H u^a - P^a_{bc} (u) \la \nabla_H u^b , \nabla_H u^c \ra \label{a3}
\end{align}
with the initial condition
\begin{align}
u^a (p, 0) = \phi^a (p), \mbox{ for all } p \in M.
\end{align}
Then $Image (u) \subset \iota(N)$ and
$$
f= \iota^{-1} (u) \in C^\infty (M \times (0, T_0), N) \cap C^0 (M \times [0, T_0),  N)
$$
solves the pseudo-harmonic flow
\begin{align}
\frac{\partial f}{\partial t}  = \tau (f) \label{a5}
\end{align}
with the initial condition
\begin{align}
f(p, 0) = \phi(p).
\end{align}
\end{theorem}

At the end of this section, we recall the CR Bochner formulas. The reader could refer to \cite{chang2013existence,greenleaf1985first,Ren201447}.

\begin{lemma}[CR Bochner Formulas, cf. \cite{Ren201447}] \label{at4}
For any smooth map $f: M \rightarrow N $, we have
\begin{align}
\frac{1}{2} \Delta_H |d_H f|^2=& |\nabla_H d_H f|^2 + \la \nabla_H \tau  (f), d_H f \ra + 4 \ii (f^i_{\ba} f^i_{0 \alpha } - f^i_\alpha f^i_{0  \ba} ) \nonumber \\
& + 2 R^M_{\alpha \bb} f^i_{\ba} f^i_{\beta} - 2\ii (n-2) (f^i_\alpha f^i_\beta A_{\ba \bb}- f^i_{\ba} f^i_{\bb} A_{\alpha \beta}  ) \nonumber \\
&+ 2 (f^i_{\ba} f^j_{\beta} f^k_{\bb} f^l_{\alpha} R^N_{ijkl} +f^i_{\alpha} f^j_{\beta} f^k_{\bb} f^l_{\ba} R^N_{ijkl}   ) \\
\frac{1}{2} \Delta_H |f_0|^2 =& |\nabla_H f_0|^2 + \la \nabla_T \tau (f) , f_0 \ra + 2 f^i_{0} f^j_{\alpha} f^k_{\ba} f^l_{0} R^N_{ijkl} \nonumber \\
& + 2 (f^i_0 f^i_\beta A_{\bb \ba, \alpha} + f^i_0 f^i_{\bb} A_{\beta \alpha, \ba} + f^i_0 f^i_{\bb \ba} A_{\beta \alpha } + f^i_0 f^i_{\beta \alpha } A_{\bb \ba} )
\end{align}
where $f^i_{AB}$ be the components of $\nabla df$ under the orthonormal coframe $ \{ \theta , \theta^\alpha, \theta^{\ba} \}$ of $T^*M$ and the orthonormal frame $\{ \xi_i \}$ of $TN$.
\end{lemma}

\begin{lemma} \label{at2}
Assume that $(N, h)$ has nonpositive sectional curvature. There exists a constant $C_1$ only depending on the bounds of the pseudo-Hermitian Ricci curvature and the pseudo-Hermitian torsion of $M$ such that for any $f \in C^\infty (M, \times (0, T_0), N) $ with $T_0 \in (0, \infty]$, we have
\begin{align}
(\Delta_H - \partial_t ) |d f|^2 \geq  - C_1 | df |^2 + |\nabla_H d_H f|^2 + |\nabla_H f_0|^2 + \la  \nabla (\tau  (f) - \partial_t f ), d f \ra.
\end{align}
In particular, if $f$ is a smooth function, then the above inequality is always true, i.e.
\begin{align}
(\Delta_H - \partial_t ) |d f|^2 \geq - C_1 | df |^2 + |\nabla_H d_H f|^2 + |\nabla_H f_0|^2 + \la  d ( \Delta_H f - \partial_t f ), d f \ra. \label{a9}
\end{align}
\end{lemma}

\begin{proof}
It suffices to prove that
\begin{align}
f^i_{\ba} f^j_{\beta} f^k_{\bb} f^l_{\alpha} R^N_{ijkl} +f^i_{\alpha} f^j_{\beta} f^k_{\bb} f^l_{\ba} R^N_{ijkl} \geq 0 \label{a6}
\end{align}
and
\begin{align}
f^i_{0} f^j_{\alpha} f^k_{\ba} f^l_{0} R^N_{ijkl} \geq 0. \label{a7}
\end{align}
Set $e_\alpha = Re \; df(\eta_\alpha)$ and $ e_{\alpha}' = Im  \; df(\eta_\alpha) $. Then
\begin{align*}
& f^i_{\ba} f^j_{\beta} f^k_{\bb} f^l_{\alpha} R^N_{ijkl} +f^i_{\alpha} f^j_{\beta} f^k_{\bb} f^l_{\ba} R^N_{ijkl} \\
&= \ \langle  R^N( df(\eta_{\bar{\beta}}) , df(\eta_\alpha) ) df(\eta_\beta), df(\eta_{\bar{\alpha}})\rangle
 +  \langle R^N( df(\eta_{\bb}), df(\eta_{\ba}) ) df(\eta_\beta), df(\eta_\alpha)\rangle \nonumber \\
&=   -2( \langle R^N( e_\alpha, e_\beta ) e_\beta, e_\alpha\rangle  + \langle R^N( e_\alpha, e_{\beta}' ) e_{\beta}', e_\alpha\rangle + \langle R^N( e_{\alpha}', e_\beta ) e_\beta, e_{\alpha}'\rangle \nonumber  \\
& \quad \quad \quad +\langle R^N( e_{\alpha}', e_{\beta}' ) e_{\beta}', e_{\alpha}'\rangle  )  \nonumber \\
&\geq 0.
\end{align*}
Hence we obtain \eqref{a6}. Similarly, to get \eqref{a7}, we calculate that
\begin{align*}
f^i_{0} f^j_{\alpha} f^k_{\ba} f^l_{0} R^N_{ijkl}  = & \langle R^N ( df(\eta_{\ba}), df(T) ) df(\eta_\alpha) , df(T)\rangle  \nonumber\\
= & -(\langle R^N (e_\alpha,e_0)e_0,e_\alpha\rangle  + \langle R^N (e_{\alpha}',e_0)e_0,e_{\alpha}'\rangle ) \nonumber \\
\geq & 0  .
\end{align*}
\end{proof}

\section{Regularity} \label{s3}
In this section, by $L^p$ theory of hypoelliptic operators, we could lift the regularity of the solution of the subelliptic parabolic system. The main theorem of this section is as follows.

\begin{theorem} \label{ct1}
Suppose that $(M, HM, J_b, \theta)$ is a closed pseudo-Hermitian manifold and $(N, h)$ is a compact Riemannian manifold which is isometrically embedded by $\iota$ in $(\mathbb{R}^K, g_{can})$. Let $B(N)$ be the tubular neighborhood of $\iota (N)$ and $P: B(N) \rightarrow \iota (N) $ the closest point projection map. Assume that $u \in C^0 (M \times (T_1, T_2), B(N) )$ for some $0 \leq T_1 < T_2 \leq +\infty $ satisfies is the weak solution of
\begin{align}
( \Delta_H - \partial_t  ) u^a = P^a_{bc} (u) \la \nabla_H u^b , \nabla_H u^c \ra.
\end{align}
If $\nabla_H u^a$ is uniformly bounded in any compact open set,
then $u \in C^\infty (M \times (T_1, T_2), B(N) )$. Moreover, if $V \Subset U \Subset M \times (T_1 , T_2)$, then for any positive integer $\gamma$
\begin{align}
\sum_a || u^a ||_{C^\gamma (V)} \leq C_{\gamma, U, V} \left( \sum_a || u^a ||_{C^0 (U)} + \sum_a || \nabla_H u^a ||_{C^0 (U)}   \right). \label{c2}
\end{align}
where $C_{\gamma, U, V}$ depends on $\gamma$, $U$ and $V$.
\end{theorem}

We firstly recall the $L^p $ theory of hypoelliptic operators (cf. \cite{rothschild1976hypoelliptic}). Let $\Omega $ be a smooth manifold and $Y_0, Y_1, \dots, Y_k$ be real smooth vector fields on $\Omega$. Define the operator
\begin{align*}
\mathcal{L} = \sum_{i=1}^k Y_i^2 + Y_0.
\end{align*}

\begin{theorem}[Theorem 1.1 in \cite{hormander1967hypoelliptic}]
Suppose the system $Y_0, Y_1, \dots, Y_k$, together with their commutators of some finite order, span the tangent space at any point of $\Omega$. Then $\mathcal{L}$ is hypoelliptic, that is if $\mathcal{L} f = g$ and $g \in C^\infty (\Omega)$, then $f \in C^\infty (\Omega)$.
\end{theorem}


\begin{example} \label{ct4}
The sub-Laplace operator $\Delta_H$ of a pseudo-Hermitian manifold $M^{2m+1}$ is hypoelliptic. Locally one can choose a horizontal orthonormal real frame $\{ X_\alpha, X_{m+\alpha} \}_{\alpha =1 }^{m}$ defined on a domain $\Omega \subset M$. Then
\begin{align}
\Delta_H = \sum_{\alpha=1}^m X_{\alpha}^2 + \sum_{\alpha=1}^m X_{m+ \alpha}^2 + \sum_{\alpha=1}^m (\nabla_{X_{\alpha}} X_\alpha + \nabla_{X_{m+ \alpha}} X_{m + \alpha}), \label{c1}
\end{align}
where $\nabla$ is the Tanaka-Webster connection. On one hand, $[HM, HM] = TM$ is from the positivity of Levi form. On the other hand, since the Webster connection preserves the horizontal bundle, the last term of \eqref{c1} is horizontal. Hence $\Delta_H$ is hypoelliptic. Thus the compatible Sobolev space is
\begin{align*}
S^p_k (\Delta_H, \Omega) = \big\{ \psi \in L^p (\Omega) \big| X_{i_1} \dots X_{i_s} \psi \in L^p (\Omega), \  s \leq k \mbox{ and } X_{i_j} \in \{ X_\alpha, X_{m+\alpha} \}_{\alpha=1}^m \big\}
\end{align*}
which is called Folland-Stein space (cf. \cite{folland1974estimates}).

The subelliptic parabolic operator
\begin{align*}
\mathcal{L}_t = \Delta_H - \partial_t
\end{align*}
is also hypoelliptic and its compatible Sobolev space is defined as follows (cf. \cite{rothschild1976hypoelliptic}): for $W \Subset \Omega \times (0, + \infty)$,
\begin{align*}
S^p_k (\mathcal{L}_t, W) = \big\{ \psi \in L^p (\Omega) \big| \partial_t^i X_{i_1} \dots X_{i_s} \psi \in L^p (W), \  2i + s \leq k \big\}.
\end{align*}
By partition of unity, we could define $S^p_k (\Delta_H, M)$ and $S^p_k (\mathcal{L}_t, M \times (T_1, T_2))$.
\end{example}

\begin{theorem}[Theorem 18 in \cite{rothschild1976hypoelliptic}] \label{ct2}
Let $\Omega$ be a region of a pseudo-Hermitian manifold and $\mathcal{L}_t = \Delta_H - \partial_t$. Suppose $f \in L^p (\Omega)$ and $\mathcal{L}_t f =g$. If $g \in S^p_k (\mathcal{L}_t , \Omega)$ for some $p >1$ and positive integer $k $, then $\chi f \in S^{p}_{k+2} (\mathcal{L}_t, \Omega)$ for each $\chi \in C^\infty_0 (\Omega) $. Moreover, there exists a constant $C_{\Omega'}$ independent of $f$ and $g$ such that
\begin{align}
|| f ||_{S^{p}_{k+2} (\mathcal{L}_t, \Omega')} \leq C_{\Omega'} ( ||f||_{L^p (\Omega)} + ||g||_{S^{p}_{k} (\mathcal{L}_t, \Omega)} ) \label{c3}
\end{align}
where $\Omega' \Subset \{ \chi=1 \}$.
\end{theorem}

\begin{theorem}[Theorem 19.2 in \cite{folland1974estimates} and Theorem 13 in \cite{rothschild1976hypoelliptic}] \label{ct3}
Suppose $(M, HM, J_b, \theta)$ is a closed pseudo-Hermitian manifold. Then 
$
S^p_\mu (\Delta_H, M) \subset L^{p}_{\mu/2} (M)
$
which is the classical Sobolev space of $M$. Moreover, for any $k \in \mathbb{N}, \alpha \in (0,1)$ and $1< p  < \infty$, there exists a sufficiently large positive integer $\mu$ such that
\begin{align*}
S^p_\mu (\mathcal{L}_t , M \times (T_1, T_2) ) \subset L^{p}_{\mu/2} (M \times (T_1, T_2)) \subset C^{k, \alpha} (M \times (T_1, T_2))
\end{align*}
where $\mathcal{L}_t = \Delta_H - \partial_t$ and $T_1 < T_2$.
\end{theorem}

The proof of the second conclusion of Theorem \ref{ct3} is trivial from the first one.
Now we can demonstrate Theorem \ref{ct1}.

\begin{proof}[Proof of Theorem \ref{ct1}]
For any point $ (q_1, t_1) \in  M \times (T_1, T_2)$, we can choose exhaustion open sets $\{ U_j \}$:
\begin{align*}
M \times (T_1, T_2) \Supset U_0 \Supset  U_1 \Supset U_2 \Supset \cdots \ni  (q_1, t_1).
\end{align*}
By the assumption, $ \nabla_H u^b  \in L^p ( U_0 )$ for any $1< p < \infty$, and thus
\begin{align*}
 P^a_{bc} (u) \la \nabla_H u^b , \nabla_H u^c \ra \in L^{p_1} ( U_0 ).
\end{align*}
for $p_1= \frac{p}{2}$.
Hence by Theorem \ref{ct2}, we have
\begin{align*}
\chi_1 u^a \in S^{p_1}_2 (\mathcal{L}_t, U_0) \Rightarrow u^a \in S^{p_1}_2 (\mathcal{L}_t, U_1)
\end{align*}
for some $\chi_1 \in C^\infty_0 ( M \times (T_1, T_2) )$ and $\chi_1 |_{U_1} \equiv 1$. Since
\begin{align*}
\big( P^a_{bc} u^b_\alpha u^c_{\ba} \big)_{\beta} = P^a_{bcd} u^d_{\beta} u^b_\alpha u^c_{\ba} + P^a_{bc} u^b_{\alpha \beta} u^c_{\ba} + P^a_{bc} u^b_\alpha u^c_{\ba \beta}
\end{align*}
then
\begin{align*}
 P^a_{bc} (u) \la \nabla_H u^b , \nabla_H u^c \ra \in S_1^{p_2} (\mathcal{L}_t, U_1).
\end{align*}
for  $p_2= \frac{p_1}{3}= \frac{p}{3 !}$.
Again by Theorem \ref{ct2},
\begin{align*}
\chi_2 u^a  \in S^{p_2}_3 (\mathcal{L}_t, U_1) \Rightarrow u^a \in S^{p_2}_3 (\mathcal{L}_t, U_2)
\end{align*}
for some $\chi_2 \in C^\infty_0 ( M \times (T_1, T_2) )$ and $\chi_2 |_{U_2} \equiv 1$.
By induction,
we have $u^a \in S^{p_k}_{k+1} (\mathcal{L}_t, U_k)$ if $p_k= \frac{ p}{( k+1) !} >1$. Theorem \ref{ct3} guarantees that for any $\gamma \in \mathbb{N}$, we can choose sufficiently large $p$ and $j$ such that $u \in C^{\gamma} (U_j)$. Hence $u$ is smooth near $(q_1, t_1)$. The estimate \eqref{c2} is due to \eqref{c3}.
\end{proof}

A similar argument show that the weak solution of the subelliptic system \eqref{a8} is also smooth. Precisely,

\begin{theorem}
Suppose that $(M, HM, J_b, \theta)$ is a closed pseudo-Hermitian manifold and $(N, h)$ is a compact Riemannian manifold. Assume that $f \in C^0 (M, N)$ is a weak pseudo-harmonic map which means $ u = \iota \circ f$ is a weak solution of \eqref{a8}. If $d_H f \in L^p (M)$ for all $1 <p < +\infty$, then $f \in C^\infty (M, N)$ and thus it is a pseudo-harmonic map.
\end{theorem}

\section{Short-Time Existence} \label{s4}
In this section, we use the heat kernel of the subelliptic parabolic operator $\Delta_H - \partial_t$ to obtain a short-time solution of \eqref{a3}.

\begin{theorem} \label{bt5}
Suppose that $(M, HM, J_b, \theta)$ is a closed pseudo-Hermitian manifold and $(N, h)$ is a compact Riemannian manifold. For any $\phi \in C^\infty (M, N)$, there exists a maximal time $ \delta >0$ such that the pseudo-harmonic flow
\begin{align}
\frac{\partial f}{\partial t}  = \tau (f) \label{b19}
\end{align}
with the initial condition
\begin{align}
f (p, 0) = \phi (p) \label{b20}
\end{align}
has a unique solution $f \in C^\infty (M \times  (0, \delta), N)  \cap C^0 (M \times [0, \delta),  N)$. Moreover, if $\delta < + \infty$, then
\begin{align*}
\liminf_{t \rightarrow \delta} || d f ( \cdot, t) ||_{C^0 (M)} = + \infty
\end{align*}
\end{theorem}

Before proving it, we recall some properties of Carnot-Carath\'eodory(CC) distance and heat kernel. Let $(M, HM, J_b, \theta)$ be a closed pseudo-Hermitian manifold with real dimension $2m +1$.

\begin{definition}
A piecewise $C^1$-curve $\gamma : [0,1] \rightarrow M$ is said to be horizontal if $\gamma' (t) \in HM$ whenever $\gamma' (t)$ exists. The length of $\gamma$ is given by $$ l(\gamma) =\int^1_0 |\gamma'|_{G_\theta} dt . $$
The Carnot-Carath\'eodory(CC) distance between two points $p,q \in M$ is
$$ d (p,q) = inf \{ l(\gamma) | \: \gamma \in C_{p,q} \} $$
where $C_{p,q}$ is the set of all horizontal curves joining $p$ and $q$.
\end{definition}

By Chow's Theorem (cf. \cite{strichartz1986sub}), there exists at least one horizontal curve arriving the CC distance. So it is finite. There are also other quasi-distances which hold weaker triangle inequality (cf. \cite{folland1974estimates,nagel1985balls}). But they are all local equivalent with CC distance. The authors in \cite{nagel1985balls} estimated the volume of the ball $B(p, \delta) = \{ q \in M | d(p, q) < \delta \}$.

\begin{lemma}[cf. Theorem 1 in \cite{nagel1985balls}] \label{bt1}
There exist constants $C_2$ and $\delta_0$ such that for any $p \in M$ and $\delta \leq \delta_0$, we have
\begin{align*}
C_2^{-1} \delta^{2m+2} \leq \int_{B(p, \delta)} d V \leq C_2 \delta^{2m +2}.
\end{align*}
\end{lemma}

Let $H(p,q ,t)$ be the fundamental solution of the subelliptic parabolic equation, i.e.
\begin{gather*}
( \partial_t - \Delta_H ) H(p,q ,t) = 0,\\
\lim_{t \rightarrow 0} H(p,q ,t) = \delta_q (p)
\end{gather*}
Hence by Duhamel's principle, the solution of
\begin{align}
( \partial_t -\Delta_H ) u (p, t) = F(p, t), \quad u(p, 0) = \phi (p) \label{b1}
\end{align}
is
\begin{align}
u(p,t)= \int_M H(p,q ,t)  \phi (q) dV_q + \int_0^t \int_M H(p,q ,t-s) F(q, s) dV_q ds \label{b2}
\end{align}
It is easy to check that $u \equiv 1$ is a solution of \eqref{b1} with $F=0$ and $\phi=1$. Hence by \eqref{b2}, we have
\begin{align}
\int_M H(p, q, t) dV_q =1.  \label{b3}
\end{align}

\begin{lemma}[\cite{jerison1987subelliptic, sanchez1984fundamental}]
$H(p, q, t)$ is positive for all $p, q \in M$ and $t >0 $. For sufficiently large $N$, there exists constant $C_{N}$ such that
\begin{align}
|  \nabla_H H(p, q, t) | \leq C_N  t^{-m-\frac{3}{2}} \bigg( 1+ \frac{d (p, q)^2 }{t} \bigg)^{-N} \label{b4} \\
H(p, q,t) \leq C_N t^{-m-1} \bigg( 1+ \frac{d (p, q)^2 }{t} \bigg)^{-N} \label{b21}
\end{align}
for $t \in (0, 1]$.
\end{lemma}

\begin{lemma} \label{bt3}
For any $\beta \in (0, \frac{1}{2})$, there exists a  constant $C_\beta$ such that
\begin{gather}
\int^t_0 \int_M |\nabla_{H, p} H(p, q , s)| dV_q ds \leq C_\beta t^{\beta} \label{b7}
\end{gather}
for any $p \in M$ and $t \in ( 0, 1 ]$. Here $\nabla_{H,p} H(p,q,s)$ means the derivative of $H(p,q,s)$ with respect to $p$. 
\end{lemma}

\begin{proof}
We choose
$$ \gamma =\frac{\beta + m + \frac{1}{2}}{2m +2} $$ 
such that
\begin{align*}
-1 < \gamma (2m+2) - (m+ \frac{3}{2}) < - \frac{1}{2}.
\end{align*}
If $q \in B (p, s^\gamma)$, then $1+ \frac{d (p, q)^2}{s} \geq 1$. If $q \in M \setminus B (p, s^\gamma )$, then $1+ \frac{d (p, q)^2}{s} \geq s^{2\gamma -1}$. 
Hence by Theorem \ref{bt1} and \eqref{b4}, we obtain
\begin{align*}
& \int^t_0 \int_M |\nabla_{H, p} H(p, q , s)| dV_q ds \\
& \leq C_N \int^t_0 \bigg( \int_{B(p, s^\gamma)} + \int_{M \setminus B(p, s^\gamma) } \bigg) s^{-m - \frac{3}{2}} \bigg( 1+ \frac{d (x, y)^2}{s} \bigg)^{- N} dV_x ds \\
& \leq C_N C_2 \int_0^t \bigg( s^{ \gamma (2m+2) -m - \frac{3}{2}} + s^{ N(1- 2 \gamma) -m - \frac{3}{2}} \bigg) ds \\
& \leq C_\beta t^\beta
\end{align*}
where $\beta =  \gamma (2m+2) -m  - \frac{1}{2}$ and $N$ is sufficiently large.
\end{proof}

A similar argument with \eqref{b21} leads to the following estimate.

\begin{lemma} \label{bt4}
For any $\alpha \in (0, \frac{1}{2})$, there exists a constant $\tilde{C}_\alpha$ such that if $\phi \in C^1 (M)$, then
\begin{gather}
\int_M  H(p, q , t) \big| \phi (q) - \phi (p) \big| dV_q \leq \tilde{C}_\alpha t^{\alpha} ( ||\phi ||_{C^0} + || \nabla_H \phi||_{C^0} ) \label{b11}
\end{gather}
for any $p \in M$ and $t \in ( 0, 1 ]$.
\end{lemma}

Analogous to the elliptic parabolic equation, the subelliptic parabolic equation has the maximum principle. There are many versions and different proofs. For convenience, we give a simple one and deduce it.

\begin{lemma} \label{bt2}
Let $u \in C^2 ( M \times (0, T_0) ) \cap C^0 ( M \times [0, T_0) )$ is a solution of the subelliptic parabolic equation
\begin{align}
(\Delta_H -  \partial_t) u \geq 0, \ u(p, t) = \phi (p) .
\end{align}
If $\phi \leq c$ for some $c \in \mathbb{R}$, then $u(p, t) \leq c$ for all $t \in [0, T_0)$ and $p \in M$.
\end{lemma}

\begin{proof}
Fix $\sigma >0$ and set $u_\sigma (p, t) = u (p, t) - \sigma (1+t)$. By the assumption, $u_\sigma (p, 0) < c$. We claim that $u_\sigma < c$ for all $t \in [0, T_0)$ and $p \in M$. Otherwise, there exists $\sigma >0$ and $(p_1, t_1) \in M \times (0, T_0)$ such that
\begin{align*}
u_\sigma \big|_{M \times [0, t_1 )} <c, \ u_\sigma(p_1, t_1) = \max_M u_\sigma (p, t_1 ) =c
\end{align*}
which yields that
\begin{align*}
\partial_t u_\sigma (p_1, t_1) \geq 0, \mbox{ and } \Delta_H u_\sigma (p_1, t_1) \leq 0.
\end{align*}
But this leads a contradiction with
\begin{align*}
(\Delta_H -  \partial_t) u_\sigma (p_1, t_1) = \sigma > 0.
\end{align*}
Hence $u_\sigma \leq c$ on $M \times [0, T_0)$. We can complete the proof by taking $\sigma \rightarrow 0$.
\end{proof}

Now we can announce the short-time existence of the subelliptic parabolic system.

\begin{theorem} \label{bt6}
Suppose that $(M, HM, J_b, \theta)$ is a closed pseudo-Hermitian manifold and $(N, h)$ is a compact Riemannian manifold which is isometrically embedded by $\iota$ in $(\mathbb{R}^K, g_{can})$. For any $\phi \in C^\infty (M, N)$, there exists $T_0 >0$ such that the subelliptic parabolic system
\begin{align}
(\Delta_H -\partial_t ) u^a=  P^a_{bc} (u) \la \nabla_H u^b , \nabla_H u^c \ra \label{b17}
\end{align}
with the initial condition
\begin{align}
u^a (p, 0) = \phi^a (p) \label{b18}
\end{align}
has a solution $u \in C^\infty (M \times  (0, T_0), B(N))  \cap C^0 (M \times [0, T_0],  B(N))$ where $B(N)$ is the tubular neighbourhood of $\iota (N)$. Moreover,
\begin{align}
\sum_a ||\nabla_H u^a||_{C^0 (M \times [0, T_0])} \leq 2 K e^{C_1} || d \phi  ||_{C^0 (M)}, \label{b9}
\end{align}
The value of $T_0$ only depends on the upper bound of $|| d \phi||_{C^0 (M)}$ and $C_1$ which is given in Lemma \ref{at2}.
\end{theorem}

\begin{proof}
The solution of the subelliptic parabolic equation
\begin{align}
(\Delta_H -\partial_t) u^a_0 =0, \mbox{ and } u^a_0 (p, 0)= \phi^a ( p ) \label{b6}
\end{align}
is given by
\begin{align*}
u^a_0 ( p , t) = \int_M H(p, q , t) \phi^a (q) d V_q.
\end{align*}
We define $u_k^a$ inductively for all $k =1, 2, \dots$ by
\begin{align*}
u^a_k (p, t) = - \int_0^t \int_M H(p, q , t-s) F^a_{k-1} (q, s) dV_q d s + u^a_0 (p,t),
\end{align*}
with
\begin{align}
F^a_{k-1} (q, s) = P^a_{bc} (u_{k-1}) \la \nabla_H u^b_{k-1} , \nabla_H u^c_{k-1} \ra (q, s) . \label{b5}
\end{align}
One can easily check that
\begin{align}
(\Delta_H -\partial_t) u^a_{k} = F^a_{k-1} \mbox{ and } u^a_k (p, 0) = \phi^a (p). \label{b16}
\end{align}
Let us consider
\begin{align*}
z_k (t) = \sup_{M \times [0, t]} \sum_a |\nabla_H u^a|
\end{align*}
which is a increasing function of $t$. Set
\begin{align*}
C_3 = \sup_{B (N), a,b, c,d} \big( |P^a_{bc} | + | P^a_{bcd} | \big).
\end{align*}
By the definition \eqref{b5} of $F_k$, we have
\begin{align*}
\sup_{M \times [0, t]} |F^a_k| \leq C_3 z_{k}^2 (t)
\end{align*}
which yields that
\begin{align}
| u^a_k (p, t) - u^a_0 (p, t) | \leq C_3 t z_{k-1}^2 (t) \label{b10}
\end{align}
and
\begin{align}
| u^a_k (p, t) | \leq C_3 t z_{k-1}^2 (t) + |u^a_0 (p, t) | \leq  C_3 t z_{k-1}^2 (t) + ||\phi^a||_{C^0 (M)}
\end{align}
due to the fact that $\int_M H(p, q, t) =1$. Moreover, Lemma \ref{bt4} shows that
\begin{align}
| u^a_0 (p, t) -\phi^a (p) | = & \left| \int_M H(p, q, t) \big( \phi^a (q) - \phi^a (p) \big) dV_q \right| \nonumber \\
\leq& \tilde{C}_\beta t^\beta ( || \nabla_H \phi^a ||_{C^0} + || \phi^a ||_{C^0} ) \label{b12}
\end{align}
for some $\beta \in (0, \frac{1}{2})$ and $\tilde{C}_\beta$ .
By \eqref{b6} and Lemma \ref{at2}, we conclude that
\begin{align*}
(\Delta_H - \partial_t) | d u^a_0 |^2 \geq - C_1 | d u^a_0 |^2,
\end{align*}
which is equivalent to
$$ (\Delta_H - \partial_t) ( e^{- C_1 t} | d u^a_0 |^2 ) \geq 0.  $$
The maximum principle (Lemma \ref{bt2}) shows that
\begin{align}
\sup_{M \times [0, t]} | d u^a_0 | \leq e^{C_1 t} || d \phi^a ||_{C^0} \Rightarrow z_0 (t) \leq K e^{C_1 t} || d \phi ||_{C^0 }. \label{b14}
\end{align}
On the other hand, by the estimate \eqref{b7}, we find
\begin{align*}
| \nabla_H u^a_k - \nabla_H u^a_0 |(p, t) \leq C_3 C_\beta t^\beta z_{k-1}^2 (t)
\end{align*}
where  $C_\beta $ is given in Lemma \ref{bt3}. Hence
\begin{align}
z_k (t) \leq K C_3 C_\beta t^\beta z_{k-1}^2 (t) + z_0 (t). \label{b8}
\end{align}
We choose that
\begin{align*}
T_1 = \min \bigg\{ \big( \frac{1}{8} K^{-2} e^{-C_1} C_3^{-1} C_\beta^{-1} || d \phi ||_{C_0}^{-1} \big)^{\frac{1}{\beta}} , 1 \bigg\} \leq 1,
\end{align*}
which yields that
\begin{align*}
K C_3 C_\beta T_1^\beta z_0 (T_1) \leq K^2 C_3 C_\beta e^{C_1} || d \phi ||_{C^0} T_1^\beta \triangleq \frac{\epsilon}{4} ,
\end{align*}
where $0 \leq \epsilon \leq  \frac{1}{2}$.
An inductive argument implies that
\begin{align*}
K C_3 C_\beta T_1^\beta z_k (T_1) \leq \frac{\epsilon}{2},
\end{align*}
since \eqref{b8} guarantees that
\begin{align*}
K C_3 C_\beta T_1^\beta z_k (T_1) \leq ( K C_3 C_\beta T_1^\beta z_{k-1} (T_1) )^2 + K C_3 C_\beta T_1^\beta z_0 (T_1) \leq \left( \frac{\epsilon}{2} \right)^2 + \frac{\epsilon}{4} \leq \frac{\epsilon}{2}.
\end{align*}
Thus we obtain \eqref{b9} because
\begin{align}
\sup_{M \times [0, T_1]} \sum_a |\nabla_H u^a_k | = z_k (T_1) \leq 2 K e^{C_1}  ||d \phi||_{C^0}. \label{b13}
\end{align}
Due to \eqref{b10} and \eqref{b12}, 
the image of $ u_k $ on $M \times [0, T_2]$ will lie in $B(N)$ by shrinking $T_1$ to $T_2$ which only depends on the upper bound of $||d \phi||_{C^0}$.

Now we show that $\{ u_k \}$ and $ \{ \nabla_H u_k \}$ are two Cauchy sequences in $C^0 (M \times [0, T_2], B(N))$. To see this, we define a non-decreasing function
\begin{align*}
X_k (t) = \sup_{M \times [0, t]} \sum_a | u^a_k -u^a_{k-1} | + \sup_{M \times [0, t]} \sum_a | \nabla_H u^a_k - \nabla_H u^a_{k-1} |.
\end{align*}
Note that
\begin{align*}
&F_k^a - F^a_{k-1} = P^a_{bc} (u_k ) \la \nabla_H u^b_k, \nabla_H u^c_k \ra - P^a_{bc} (u_{k-1} ) \la \nabla_H u^b_{k-1}, \nabla_H u^c_{k-1} \ra  \\
&= ( P^a_{bc} (u_k )- P^a_{bc} (u_{k-1}) ) \la \nabla_H u^b_k, \nabla_H u^c_k \ra + P^a_{bc} (u_{k-1} ) \la \nabla_H u^b_k - \nabla_H u^b_{k-1 }, \nabla_H u^c_k \ra \\
& \quad + P^a_{bc} (u_{k-1} ) \la \nabla_H u^b_{k-1}, \nabla_H u^c_k - \nabla_H u^c_{k-1} \ra
\end{align*}
Hence by \eqref{b13} and the mean value theorem, we have
\begin{align*}
\sup_{M \times [0, t ]} | F_k^a - F^a_{k-1} | \leq &  C_3 X_k (t) \big( z^2_k (t) + z_k (t) + z_{k-1} (t) \big) \leq  C_4 X_k (t),
\end{align*}
where $C_4 = 4 C_3 e^{2C_1} K^2 \big( ||d \phi||_{C^0}^2 + ||d \phi||_{C_0} \big)$. The fact that $\int_M H(p, q , t) =1$ and Lemma \ref{bt3} lead that
\begin{align*}
| u^a_k - u^a_{k-1} | \leq \int_0^t \int_M H(p, q , t-s) | F_k^a - F^a_{k-1} | (q, s) d V_q ds \leq C_4 t X_{k-1} (t)
\end{align*}
and
\begin{align*}
| \nabla_H u^a_k - \nabla_H u^a_{k-1} | \leq & \int_0^t \int_M | \nabla_{H, p} H(p, q , t-s)| \: | F_k^a - F^a_{k-1} | (q, s) d V_q ds \\
\leq & C_4 C_\beta t^\beta X_{k-1} (t).
\end{align*}
For simplicity, we assume $C_\beta \geq 1$. Otherwise we set $C_\beta =1$. Thus
\begin{align}
X_k (t) \leq 2 K C_4 C_\beta t^\beta X_{k-1} (t). \label{b15}
\end{align}
for all $k \geq 2$. For $k=1$, by \eqref{b14}, we have
\begin{align*}
| u^a_1 - u^a_0 | \leq \int_0^t \int_M H(p, q , t-s) |F^a_0 (q, s)| d V_q ds \leq C_3 t z_0^2 (t) \leq K^2 C_3 t e^{2 C_1} ||d \phi||^2_{C^0},
\end{align*}
and
\begin{align*}
| \nabla_H u^a_1 - \nabla_H u^a_{0} | \leq & \int_0^t \int_M | \nabla_{H, p} H(p, q , t-s)| \: | F_0^a  (q, s)| d V_q ds \\
\leq & C_3 C_\beta t^\beta z_0^2 (t) \leq K^2 C_3 C_\beta t^\beta e^{2 C_1} ||d \phi||^2_{C^0}.
\end{align*}
Hence we obtain the estimate
\begin{align*}
X_1 (t) \leq 2 K^2 C_3 C_\beta t^\beta e^{2 C_1} ||d \phi||^2_{C^0}.
\end{align*}
By \eqref{b15} and shrinking $T_2$ to $T_0 = \min \{ T_2, (\frac{1}{4} K C_4 C_\beta )^{\frac{1}{\beta}} \}$, we can conclude that
\begin{align*}
X_k (T_0) \leq \bigg( \frac{1}{2} \bigg)^{k-1} X_1 (T_0 ) \leq \bigg( \frac{1}{2} \bigg)^{k-2} K^2 C_3 C_\beta e^{2 C_1} ||d \phi||^2_{C^0}.
\end{align*}
which yields that
\begin{align*}
\sup_{M \times [0, T_0]} \sum_a \big( | u^a_k - u^a_l| + | \nabla_H u^a_k - \nabla_H u^a_l | \big) \leq K^2 C_3 C_\beta e^{2 C_1} ||d \phi||^2_{C^0} \sum_{i=k+1}^l \bigg( \frac{1}{2} \bigg)^{i-2}.
\end{align*}
Hence there exists the convergence $u \in C^0 ( M \times [0, T_0], B(N)) $ of $\{ u_k \}$ such that $\nabla_H u_k $ uniformly converges to $\nabla_H u$ on $M \times [0, T_0]$. Thus the conclusion \eqref{b9} follows from \eqref{b13}. Moreover,
\begin{align*}
F^a_k \rightarrow F^a = P^a_{bc} (u) \la \nabla_H u^b, \nabla_H u^c \ra.
\end{align*}
Taking $k \rightarrow +\infty$ on \eqref{b16}, $u$ is a weak solution of the subelliptic parabolic system \eqref{b17} and \eqref{b18}. According to the regularity theorem \ref{ct1} and the estimate \eqref{b9}, $ u \in C^\infty (M \times (0, T_0) , B (N))$. This completes the proof.
\end{proof}

To prove Theorem \ref{bt5}, we also need the uniqueness theorem.

\begin{theorem}[Uniqueness] \label{bt7}
Assume that $u, v \in C^\infty (M \times  (0, T_0), B(N))  \cap C^0 (M \times [0, T_0],  B(N))$ are two solutions of the subelliptic parabolic system \eqref{b17} with the same initial condition. If $\nabla_H v$ and $\nabla_H u$ are uniformly bounded on $M \times (0, T_0)$, then $u \equiv v $.
\end{theorem}

\begin{proof}
It suffices to prove the function $w = \sum_a (u^a - v^a)^2 $ vanishes. Note that
\begin{align*}
& \big| (\Delta_H - \partial_t) (u^a - v^a) \big| = | F^a (u) - F^a (v) |  =  |P^a_{bc} (u) \la \nabla_H u^b, \nabla_H u^c \ra - P^a_{bc} (v) \la \nabla_H v^b, \nabla_H v^c| \ra  \\
& \leq | (P^a_{bc} (u) -P^a_{bc} (v) ) \la \nabla_H u^b, \nabla_H u^c \ra  | + | P^a_{bc} (v) \la \nabla_H u^b -\nabla_H v^b, \nabla_H v^c \ra |  + | P^a_{bc} (v) \la \nabla_H u^b, \nabla_H u^c -\nabla_H v^c \ra | \\
& \leq C_5 w^{\frac{1}{2}} + C_5' \sum_b | \nabla_H u^b -\nabla_H v^b |
\end{align*}
where $C_5$ and $C_5'$ depends on the bounds of $\nabla_H u^a$, $\nabla_H v^a$, $P^a_{bc}$ and $P^a_{bcd}$.
Hence by Cauchy inequality, we have
\begin{align*}
(\Delta_H - \partial_t) w =& 2 \sum_a (u^a - v^a) (\Delta_H - \partial_t) (u^a - v^a) + 2 \sum_a |\nabla_H u^a - \nabla_H v^a|^2 \\
\geq & - C_5'' w
\end{align*}
The maximum principle (Lemma \ref{bt2}) shows that
\begin{align*}
0 \leq w(p, t) \leq e^{C_5'' t} \sup_M |w (\cdot, 0)| =0
\end{align*}
which yields $u \equiv v$.
\end{proof}

Now we deduce Theorem \ref{bt5}.

\begin{proof}[Proof of Theorem \ref{bt5}]
Theorem \ref{at3} and Theorem \ref{bt6} guarantee the existence of the pseudo-harmonic flow \eqref{b19} with the initial condition \eqref{b20} near $t=0$. Now we prove that if $\delta < +\infty $ is the maximal time, then
\begin{align*}
\liminf_{ t \rightarrow \delta} || d u (t, \cdot)||_{C^0 (M)} = \infty.
\end{align*}
Otherwise, there exists a sequence $t_n$ and a positive number $B$ such that $t_n \rightarrow \delta$ and
\begin{align*}
|| d  u (\cdot, t_n)||_{C^0 (M)} \leq B, \mbox{ for each } n.
\end{align*}
We choose $ T_0 $ as in Theorem \ref{bt6} which makes the solution extend to $[t_n, t_n + T_0]$. Theorem \ref{bt7} ensures that the extended solution coincides with the original one $u$ on $[t_n,  \delta)$. This leads a contradiction with the definition of $\delta$.
\end{proof}

\section{Long-Time Existence} \label{s1}
For the long-time existence, it suffices to prove the total energy density is uniformly bounded. In the harmonic case, the energy is nonincreasing (cf. \cite{eells1964harmonic}) and thus is uniformly bounded; so is the energy density by the Morse iteration. This process has be generalized to pseudo-harmonic maps under the analytic assumption $[\Delta_H, T] =0$ (cf. \cite{chang2013existence}) where $T$ is the Reeb vector field. Note that Sasakian manifolds hold the assumption. But in general, this may be not true. However, some special tricks come into play and make the total energy also bounded.

Suppose that $(M, HM, J_b, \theta)$ is a closed pseudo-Hermitian manifold and $(N, h)$ is a compact Riemannian manifold with nonpositive sectional curvature. Let $\psi: M \rightarrow N$ be a smooth map.
Define the Reeb energy of $\psi $ by
\begin{align*}
E_R (\psi) = \int_M e_R (\psi) \theta \wedge (d \theta)^m
\end{align*}
where $e_R (\psi) = \frac{1}{2} |d\psi(T)|^2 = \frac{1}{2} |\psi_0|^2$ is called the Reeb energy density. The total energy and the total energy density are respectively given by
\begin{align*}
E(\psi)= \int_M e(\psi) = E_H(\psi) + E_R(\psi),
\end{align*}
and
\begin{align*}
e(\psi) = e_H (\psi) + e_R (\psi).
\end{align*}

Assume that $f : M \times (0, \delta) \rightarrow  N$ is a smooth map and satisfies the pseudo-harmonic flow
$$
\frac{\partial f}{\partial t}= \tau  (f).
$$
Set $f_t = f (\cdot, t)$. To deal with the total energy $E (f_t)$, we consider the horizontal part $E_H (f_t)$ and the Reeb part $E_R (f_t)$ respectively.

\begin{lemma}  \label{dt1}
$E_H (f_t)$ is a convex decreasing smooth function. Precisely
\begin{align}
\frac{d}{dt} E_H (f_t) =& - \int_M |\partial_t f|^2 = - \int_M |\tau  (f_t)|^2  \leq 0 \label{d1} \\
\frac{d^2}{d t^2 } E_H (f_t) =& 2 \int_M | \nabla_H \partial_t f |^2 - 2 \int_M trace_{G_\theta} \la R^N(\partial_t f,  d_H f) d_H f , \partial_t f \ra \geq 0 \label{d2}
\end{align}
Hence $E_H (f_t) \leq E_H (f_0) $. Moreover if the maximal time $\delta = + \infty$ , then $|| \tau  (f_t) ||_{L^2(M)} \rightarrow 0$ as $t \rightarrow \infty$.
\end{lemma}

\begin{proof}
Since $\nabla d f (\partial_t, X) = \nabla d f (X, \partial_t)$ for any spatial vector field $X$, \eqref{d1} follows from
\begin{align*}
\frac{d}{dt} E_H (f_t) = \int_M \la \nabla_{\partial_t} d_H f, d_H f \ra = \int_M \la \nabla_H \partial_t f, d_H f \ra =- \int_M |\partial_t f|^2.
\end{align*}
Hence
\begin{align}
\frac{d^2}{d t^2 } E_H (f_t)= & \int_M \partial_t \la \nabla_H \partial_t f, d_H f \ra = \int_M \partial_t \la \nabla_H \partial_t f, d_H f \ra \nonumber\\
=& \int_M |\nabla_H \partial_t f|^2 + \int_M \la \nabla_{\partial_t} \nabla_H \partial_t f, d_H f  \ra \nonumber \\
=& \int_M |\nabla_H \partial_t f|^2 + \int_M trace_{G_\theta} \la R^N(\partial_t f,  d_H f) \partial_t f, d_H f \ra \nonumber \\
& +\int_M \la \nabla_H \nabla_{\partial_t} \partial_t f, \nabla_H f \ra \label{d3}
\end{align}
Note that
\begin{align*}
\mbox{the last term of \eqref{d3}}=& - \int_M \la  \nabla_{\partial_t} \partial_t f, \tau  (f) \ra \\
=& - \frac{1}{2} \frac{d}{ dt} \int_M |\partial_t f|^2 =\frac{1}{2} \frac{d^2}{d t^2 } E_H (f_t)
\end{align*}
which yields \eqref{d2}.
\end{proof}

\begin{lemma}
Given $t_0 \in (0, \delta)$. Then for any $t \in (t_0, \delta)$, we have
\begin{align}
E_R (f_t ) \leq \frac{1}{2m} \bigg( ||\tau  (f_{t_0})||_{L^2}^2 + C_6 E_H (f_{t_0}) \bigg) + E_R ( f_{t_0} ) \: e^{2m (t_0 -t)} \label{d14}
\end{align}
where $C_6$ only depends on the bounds of the pseudo-Hermitian Ricci curvature and the pseudo-Hermitian torsion.
\end{lemma}

\begin{proof}
Lemma \ref{at4} implies that
\begin{align*}
(\partial_t - \Delta_H) e(f_t) \leq& - |\nabla_H d_H f|^2 - | \nabla_H f_0 |^2 - 4 \ii (f^i_{\ba} f^i_{0 \alpha } - f^i_\alpha f^i_{0  \ba} ) \\
& + 2 R^M_{\alpha \bb} f^i_{\ba} f^i_{\beta} - 2\ii (n-2) (f^i_\alpha f^i_\beta A_{\ba \bb}- f^i_{\ba} f^i_{\bb} A_{\alpha \beta}  ) \\
& - 2 (f^i_0 f^i_\beta A_{\bb \ba, \alpha} + f^i_0 f^i_{\bb} A_{\beta \alpha, \ba} + f^i_0 f^i_{\bb \ba} A_{\beta \alpha } + f^i_0 f^i_{\beta \alpha } A_{\bb \ba} ).
\end{align*}
Integrating on $M$ and using the divergence theorem, the result is
\begin{align}
\frac{d}{d t} E(f_t) \leq& - \int_M |\nabla_H d_H f|^2 - \int_M | \nabla_H f_0 |^2- 4 \ii \int_M (f^i_{\ba} f^i_{0 \alpha } - f^i_\alpha f^i_{0  \ba} )  \nonumber \\
& + 2 \int_M R^M_{\alpha \bb} f^i_{\ba} f^i_{\beta} - 2\ii (n-2) \int_M (f^i_\alpha f^i_\beta A_{\ba \bb}- f^i_{\ba} f^i_{\bb} A_{\alpha \beta}  ) \nonumber \\
& + 2 \int_M (f^i_{0 \alpha} f^i_\beta A_{\bb \ba} + f^i_{0  \ba} f^i_{\bb} A_{\beta \alpha} ). \label{d4}
\end{align}
By Cauchy inequality, \eqref{d4} becomes
\begin{align}
\frac{d}{d t} E(f_t) \leq - \int_M |\nabla_H d_H f|^2 + C_{6} E_H (f_t). \label{d5}
\end{align}
On the other hand, the commutate relation (cf. \cite{greenleaf1985first,lee1988psuedo,Ren201447})
\begin{align}
f^i_{\alpha \bb} -f^i_{\bb \alpha}  = & 2 \ii f^i_0 \delta_{\alpha \bb} \label{d13}
\end{align}
implies that
\begin{align}
|\nabla_H d_H f|^2 =& 2 \sum_{\alpha, \beta=1}^m (f^i_{\ba \beta} f^i_{\alpha \bb} + f^i_{\alpha \beta} f^i_{\ba \bb}  ) \geq 2 \sum_{\alpha=1}^m f^i_{\alpha \ba} f^i_{\ba \alpha} \nonumber \\
= & \frac{1}{2} \sum_{\alpha=1}^m \big[ |f^i_{\alpha \ba} +f^i_{\ba \alpha} |^2 + |f^i_{\alpha \ba} - f^i_{\ba \alpha} |^2  \big] \geq \frac{1}{2} \sum_{\alpha=1}^m |f^i_{\alpha \ba} -f^i_{\ba \alpha} |^2  \nonumber \\
=& 2m |f_0|^2.
\end{align}
Hence
\begin{align}
\frac{d}{d t} E(f_t) \leq - 2m E_R (f_t) + C_{6} E_{H} (f_t). \label{d6}
\end{align}
Lemma \ref{dt1} guarantees that $\frac{d}{dt} E_H (f_t)$ is increasing and nonpositive. Hence
\begin{align*}
0 \geq \frac{d}{dt} E_H (f_t)  \geq \frac{d}{dt} E_H (f_{t_0}) = - ||\tau  (f_{t_0})||_{L^2}^2.
\end{align*}
By using Lemma \ref{dt1} again, \eqref{d6} becomes
\begin{align}
\frac{d}{d t} E_R (f_t) + 2m E_R(f_t) \leq ||\tau  (f_{t_0})||_{L^2}^2 + C_{6} E_{H} (f_0) .
\end{align}
which implies \eqref{d14}.
\end{proof}

Take $t_0 = \frac{T_0}{2}$ where $T_0$ is given by Theorem \ref{bt6}. The estimates \eqref{c2} and \eqref{b9} implies the $C^2$-norm of the solution given by Theorem \ref{bt5}. Then we have the uniform estimate of the total energy $E (f_t)$.

\begin{lemma}
Let $f: M \times (0, \delta) \rightarrow N$ be the solution of the pseudo-harmonic heat flow with the initial condition $\phi = (\phi^1, \dots, \phi^K)$. Then for any $t \in ( \frac{T_0}{2}, \delta)$,
\begin{align}
E (f_t) \leq C_7 \bigg( ||d \phi||_{C^0}^2 + \sum_a ||\phi^a||_{C^0}^2 \bigg)  \label{d10}
\end{align}
where $\delta$ is the maximal time and $C_7$ only depends on the upper bound of $||d \phi||_{C^0}$.
\end{lemma}

Lemma \ref{at2} shows that the total energy density $e(f_t)$ satisfies the subelliptic parabolic inequality
\begin{align}
(\Delta_H - \partial_t ) e(f_t) \geq - C_1 e(f_t). \label{d7}
\end{align}
In \cite{moser1964harnack}, Moser gave his famous method, called Moser Iteration, to estimate the $L^\infty$-norm by the $L^2$-norm for the positive subsolution of the elliptic parabolic version of \eqref{d7}. The only ingredient of Moser Iteration is Sobolev inequality. The CR version of Sobolev inequality is also valid.

\begin{theorem}[Theorem 3.13 in \cite{dragomir2006differential} and \cite{folland1974estimates}]
$S^2_1 (\Delta_H, M) \subset L^{\frac{2m +2}{m}} (M)$.
\end{theorem}

Hence we can repeat Moser Iteration and obtain the next lemma.

\begin{lemma} \label{dt2}
Assume that $\psi \in  C^\infty ( M \times (0, \delta) )$ is nonnegative and satisfies
\begin{align}
(\Delta_H -\partial_t ) \psi \geq 0. \label{d8}
\end{align}
Then for any $\epsilon \in (0, \delta)$ and $t \in [\epsilon, \delta )$, we have
\begin{align}
\psi (p, t) \leq C_\epsilon \int_{t- \epsilon}^t \int_M \psi (q, s) d V_q ds,  \label{d9}
\end{align}
where $C_\epsilon $ only depends on $\epsilon$.
\end{lemma}


Now we come back to \eqref{d7}. It can be rewritten as
\begin{align*}
(\Delta_H - \partial_t ) ( e^{-C_1 t} e (f_t) ) \geq 0,
\end{align*}
which, by Lemma \ref{dt2} and choosing $\epsilon = \frac{T_0}{2}$, implies that for any $t \in [ \frac{T_0}{2} , \delta)$,
\begin{align*}
e (f_t) (p, t) \leq  C_\epsilon C_1^{-1} e^{ \frac{1}{2} C_1 T_0 } \sup_{s \in [t- \frac{T_0}{2}, t]} E(f_s).
\end{align*}
By \eqref{d10}, we conclude that for any $t \in (\frac{T_0}{2}, \delta)$,
\begin{align}
|| d f_t ||_{C^0} \leq \tilde{C}_8 \bigg( ||d \phi||_{C^0} + \sum_a ||\phi^a||_{C^0} \bigg) \label{d11}
\end{align}
where $C_8$ only depends on the upper bound of $||d \phi||_{C^0}$.

\begin{theorem} \label{dt3}
Suppose that $(M, HM, J_b, \theta)$ is a closed pseudo-Hermitian manifold and $(N, h)$ is a compact Riemannian manifold. For any $\phi \in C^\infty (M, N)$, the pseudo-harmonic heat flow
\begin{align}
\frac{\partial f}{\partial t}  = \tau (f)
\end{align}
with the initial condition
\begin{align}
f (p, 0) = \phi (p)
\end{align}
has a unique solution $f \in C^\infty (M \times (0, +\infty), N)  \cap C^0 (M \times [ 0, +\infty ),  N)$. Moreover, we have the estimates
\begin{align}
\sup_{(0, +\infty)} || d_H f_t ||_{C^0} \leq C_8 \bigg( ||d \phi||_{C^0} + \sum_a ||\phi^a||_{C^0} \bigg),
\end{align}
and
\begin{align}
\sup_{[ \frac{T_0}{2}, + \infty )} ||  f_t ||_{C^\gamma} \leq C_\gamma \bigg( ||d \phi||_{C^0} + \sum_a ||\phi^a||_{C^0} \bigg), \label{d12}
\end{align}
where $C_8$ only depends on the upper bound of $||d \phi||_{C^0}$, $C_\gamma$ depends on $\gamma$ and $t_0$.
Thus there exists a sequence $t_n$ and a pseudo-harmonic map $f_\infty: M \rightarrow N$ such that $f_{t_n} \rightarrow f_\infty$ in $C^\infty (M, N) $ as $t_n \rightarrow +\infty$.
\end{theorem}

\begin{proof}
Since $\delta$ is the maximal time, it must be $+ \infty$ by Theorem \ref{bt5}. Due to \eqref{c2} and \eqref{d11}, the derivatives of order $\leq k$ of the solution $f$ are uniformly bounded. The estimate \eqref{d12} follows from \eqref{c2}. Then we complete the proof by the Arzel\`a-Ascoli theorem.
\end{proof}

\section{Homotopy Class of Pseudo-Harmonic Maps} \label{s5}
In this section, we firstly deduce that the solution of the pseudo-harmonic heat flow is smoothly dependent of the initial data. As a consequence, the solution $f_t$ in Theorem \ref{dt3} is uniformly convergent to $f_\infty$. Secondly, we consider that the target manifold $(N, h)$ has negative sectional curvature. In this case, if the image of the pseudo-harmonic map $f_\infty $ is neither a point nor a closed geodesic, then it is the unique pseudo-harmonic map in the homotopic class of $f_\infty $.

Suppose that $(M, HM, J_b, \theta)$ is a closed pseudo-Hermitian manifold and $(N, h)$ is a compact Riemannian manifold. Let $\phi \in C^\infty (M \times I, N)$ where $I$ is an open domain in $\mathbb{R}$. By Theorem \ref{dt3}, for any fixed $s \in I$, the pseudo-harmonic flow
\begin{align}
\frac{\partial K}{\partial t}  = \tau (K)
\end{align}
with the initial condition
\begin{align}
K (p, 0; s) = \phi (p; s)
\end{align}
has a unique solution $K(p, t; s) \in C^\infty (M \times (0, +\infty), N)  \cap C^0 (M \times [ 0, +\infty ),  N)$.

\begin{lemma} \label{e1}
$K \in C^\infty (M \times (0, +\infty) \times I, N) $.
\end{lemma}

According to Theorem \ref{dt3}, all derivative of $K$ along the spatial vector fields are inner-closed uniformly bounded; so are the derivative of the time variable. Hence, by the next lemma, it suffices to prove $\partial_s^k K$ are inner-closed uniformly bounded.


\begin{lemma}[Lemma 6.2 in \cite{rothschild1976hypoelliptic}]
Suppose that a complex valued function $ (\gamma, x) \rightarrow F(\gamma, x) $ is defined on an open subset $U \subset \mathbb{R}^{n_1} \times \mathbb{R}^{n_2}$ and satisfies the following properties:
\begin{enumerate}[(i)]
    \item F is $C^{\infty}$ in both variables $\gamma$ and $x$, separately,
    \item all partial derivatives of $F$ with respect to either $x$ or $\gamma$ are bounded on compact subset  of $U$.
\end{enumerate}
Then $F$ is jointly $C^\infty $ on $U$.
\end{lemma}

\begin{proof}[Proof of Lemma \ref{e1}]
We only prove that $\partial_s K$ is inner-closed uniformly bounded. The same method could deduce the higher derivatives by inductive.
Denote $K_s (p, t) = K (p, t; s)$. The target manifold $(N, h)$ can be isometrically embedded by $\iota$ in $(\mathbb{R}^K, g_{can})$. Let $B(N)$ be the tubular neighborhood of $\iota (N)$ and $P: B(N) \rightarrow \iota (N) $ the closest point projection map. Let $u = \iota \circ K$ satisfying
\begin{align}
(\Delta_H - \partial_t) u^a =  P^a_{bc} (u) \la \nabla_H u^b , \nabla_H u^c \ra
\end{align}
with the initial condition
\begin{align}
u^a (p, 0; s) = \phi^a (p; s).
\end{align}
Moreover, we have the estimates
\begin{align}
\sup_{(t, s) \in (0, +\infty) \times I} || d_H u_s (\cdot , t) ||_{C^0} \leq C_8 \bigg( ||d \phi||_{C^0} + \sum_a ||\phi^a||_{C^0} \bigg), \label{e2}
\end{align}
and
\begin{align}
\sup_{(t, s) \in [ \frac{T_0}{2}, + \infty ) \times I } ||  u_s^a (\cdot , t) ||_{C^\gamma} \leq C_{\gamma } \bigg( ||d \phi||_{C^0} + \sum_a ||\phi^a||_{C^0} \bigg). \label{e3}
\end{align}

Firstly, we examine that $u (p, t ;s)$ is Lipschitz in $s$. Fixed $s_0 \in I$. Let $v (p, t; s) = u (p, t; s+ s_0) - u (p, t; s_0)$ which satisfies
\begin{align}
v^a_s (p, t ) = - \int_0^t \int_M H(p, q, t-\lambda) G^a_s (q, \lambda) dV_q d \lambda + \Phi_s^a (p, t), \label{e10}
\end{align}
where
\begin{align*}
\Phi_s^a (p, t) = \int_M H(p, q, t) \big( \phi^a_{s+s_0} (q) - \phi^a_{s_0} (q) \big) d V_q,
\end{align*}
and
\begin{align*}
G^a_s= & ( P^a_{bc} (u_{s+s_0} )- P^a_{bc} (u_{s_0}) ) \la \nabla_H u^b_{s + s_0}, \nabla_H u^c_{s + s_0} \ra \\
& \quad + P^a_{bc} (u_{s_0} ) \la \nabla_H u^b_{s + s_0} - \nabla_H u^b_{s_0 }, \nabla_H u^c_{s + s_0} \ra \\
& \quad + P^a_{bc} (u_{s_0} ) \la \nabla_H u^b_{s_0}, \nabla_H u^c_{s + s_0} - \nabla_H u^c_{s_0} \ra.
\end{align*}
Denote
\begin{align*}
V_s (t) = \sup_{M \times [0,t]} \sum_a \big( |v^a_s| + | \nabla_H v^a_s | \big).
\end{align*}
By \eqref{e2} and the mean value theorem, we have $| G^a_s (p,t) | \leq C_9  V_s (t)$. Moreover, a similar argument as \eqref{b14} shows that $ || \nabla \Phi^a_s ( \cdot, t ) ||_{C^0} \leq e^{C_1 t} || \nabla \phi^a_{s + s_0} - \nabla \phi^a_{s_0} ||_{C^0} = O(s)$. Hence combining with Lemma \ref{bt3}, we find
\begin{align*}
V_s (t) \leq C_\beta' t^\beta V_s (t) + e^{C_1 t} || \phi_{s+ s_0} - \phi_s ||_{C^1}.
\end{align*}
Choosing $T_3 \leq 1$ such that $C_{\beta}'  T_3^\beta \leq \frac{1}{2}$, we obtain that
\begin{align}
\sup_{M \times [0, T_3]} \sum_a \big( |v^a_s| + | \nabla_H v^a_s | \big) = V_s (T_3)  \leq 2 e^{C_1 } || \phi_{s + s_0} - \phi_{s_0} ||_{C^1} = O(s), \label{e8}
\end{align}
which, by Theorem \ref{ct2} and Theorem \ref{ct3}, yields that
\begin{align}
\sum_a \big( ||v^a_s  ||_{C^0} + || \nabla v^a_s||_{C^0} \big) \big|_{t =  \frac{T_3}{2} } \leq C_9' || \phi_{s + s_0} - \phi_{s_0} ||_{C^1} = O(s). \label{e4}
\end{align}
On the other hand, since $( \Delta_H - \partial_t ) v_s^a = G^a_s $, by \eqref{a9}, Cauchy inequality and the commutation relation $ (v^a_s)_{0 \alpha}- (v^a_s)_{\alpha 0 } = (v^a_s)_{\bb} A_{\beta \alpha} $ (cf. \cite{greenleaf1985first,lee1988psuedo,Ren201447}), we know that
\begin{align*}
(\Delta_H - \partial_t) \sum_a \big( |v_s^a|^2 + | \nabla v^a_s |^2 \big)  \geq - C_9''  \sum_a \big( |v_s^a|^2 + | \nabla v^a_s |^2 \big).
\end{align*}
The maximum principle (Lemma \ref{bt2}) shows that for $t \geq \frac{T_3}{2}$,
\begin{align}
\sum_a \big( |v_s^a|^2 + | \nabla v^a_s |^2 \big) (p, t) \leq& e^{C_9'' ( t - \frac{T_3}{2} )}  \sup_{ q \in M} \sum_a \bigg( |v_s^a (q , \frac{T_3}{2}) |^2 + | \nabla v^a_s (q , \frac{T_3}{2}) |^2 \bigg) \nonumber \\
=&  O(s), \label{e9}
\end{align}
due to \eqref{e4}. Hence $ v_s= O (s)$, i.e. $u (p, t ;s)$ is Lipschitz in $s$.

Secondly, we consider the system
\begin{align}
(\Delta_H - \partial_t) w^a_s = Q^a_s, \mbox{ with } w^a_s (p, 0) = \partial_s \phi^a_s (p), \label{e5}
\end{align}
where
\begin{align*}
Q^a_s =& P^a_{bcd}(u_s) \la \nabla_H u_s^b, \nabla_H u_s^c \ra w^d_s + P^a_{bc} (u_s) \la \nabla_H w^b_s, \nabla_H u^c_s \ra \\
& + P^a_{bc} (u_s) \la \nabla_H u^b_s, \nabla_H w^c_s \ra.
\end{align*}
The short-time existence is similar to Theorem \ref{bt6}, and by the maximum principle, the long-time existence follows from the estimate that for $ t \geq \frac{T_3}{2}$
\begin{align*}
(\Delta_H - \partial_t) \sum_a \big( |w_s^a|^2 + | \nabla w^a_s |^2 \big)  \geq - C_{10}  \sum_a \big( |w_s^a|^2 + | \nabla w^a_s |^2 \big)
\end{align*}
due to \eqref{a9}. Moreover, the two process show that
\begin{align}
\sup_{(t, s) \in (0, +\infty) \times I} \sum_a \big( ||w^a_s ||_{C^0} + || d_H w_s^a  ||_{C^0} \big) \leq C_{10}' \sum_a \bigg( ||\partial_s \phi^a_s||_{C^0} + ||d \partial_s \phi^a_s ||_{C^0}  \bigg), \label{e7}
\end{align}
and
\begin{align}
\sup_{(t, s) \in [ \frac{T_0}{2}, + \infty ) \times I } \sum_a || w_s^a (\cdot , t) ||_{C^\gamma} \leq C_{\gamma }' \sum_a \bigg( ||\partial_s \phi^a_s||_{C^0} + ||d \partial_s \phi^a_s||_{C^0}  \bigg). \label{e6}
\end{align}

Thirdly, to prove that $\partial_s u_s$ exists, it suffices to demonstrate that
$$
y_s = (u_{s+ s_0} - u_{s_0}) - w_{s_0} s = v_s - w_{s_0 } s = o (s),
$$
which implies that $\partial_s u_s = w_s $ is inner-closed uniformly bounded by \eqref{e6}. After a tedious calculation and using the estimates \eqref{e8}, \eqref{e9}, we conclude that
\begin{align*}
(\Delta_H - \partial_t ) y^a_s =&  G^a_s - Q^a_{s_0 } s \\
=& P^a_{bcd}(u_{s_0 }) \la \nabla_H u_{s_0 }^b, \nabla_H u_{s_0 }^c \ra y^d_s + P^a_{bc} (u_{s_0 }) \la \nabla_H y^b_s, \nabla_H u^c_{s_0 } \ra \\
& + P^a_{bc} (u_{s_0 }) \la \nabla_H u^b_{s_0 }, \nabla_H y^c_s \ra + O(s^2)
\end{align*}
with the initial condition $y_s (p, 0) = o(s)$.
As the argument for $v_s$, we can estimate $y_s = o (s)$ on $t \leq T_3$ by the expression analogous to \eqref{e10} and on $t \geq \frac{T_3}{2}$ by the CR Bochner formula \eqref{a9} and the maximum principle. Since the details are the repetition of the above process, we omit them.
\end{proof}

Now we choose a local unitary frame $\{ \eta_\alpha \}$ of $T_{1, 0} M$ with $\nabla^M  \eta_\alpha  =0$ at a given point $p \in M$ and do the calculation at $p$,
\begin{align*}
\frac{1}{2} \Delta_H | \partial_s K |^2 =& | \nabla_H \partial_s K |^2 + \la \nabla_{ \eta_{\ba } } \nabla_{ \eta_\alpha } \partial_s K, \partial_s K \ra + \la \nabla_{ \eta_{\alpha } } \nabla_{ \eta_{\ba} } \partial_s K, \partial_s K \ra \\
= & | \nabla_H \partial_s K |^2 + \la \nabla_{ \eta_{\ba } } \nabla_{ \partial_s } dK (\eta_\alpha), \partial_s K \ra + \la \nabla_{ \eta_{\alpha } } \nabla_{ \partial_s } dK ( \eta_{\ba} ), \partial_s K \ra \\
= & | \nabla_H \partial_s K |^2 + \la \nabla_{\partial_s} \tau (K), \partial_s K \ra + trace_{G_\theta} \la R^N ( d_H K, \partial_s K) d_H K ,  \partial_s K \ra
\end{align*}
Hence we have
\begin{align}
\frac{1}{2} ( \Delta_H -\partial_t) | \partial_s K |^2 = | \nabla_H \partial_s K |^2  + trace_{G_\theta} \la R^N ( d_H K, \partial_s K) d_H K ,  \partial_s K \ra. \label{e13}
\end{align}
By the maximum principle, we obtain the following lemma.

\begin{lemma} \label{et1}
Assume that $(N,h)$ has nonpositive sectional curvature.
Suppose that $ \phi(p ; s) \in C^\infty (M \times I ,N) $ and $K(p, t; s)$ is the solution of the pseudo-harmonic flow with initial data $\phi (p ; s)$ for fixed $s$. Then for any $ s \in I$,
$|| \partial_s K (\cdot , t ; s) ||_{C^0 (M)} $
is non-increasing in $t$.
\end{lemma}

As a consequence of Lemma \ref{et1}, the distance between homotopic solutions of the pseudo-harmonic flow must be nonincreasing. The distance $d (f_0, f_1)$ of two homotopic maps $f_0$ and $f_1$ is defined as follows. If $F: M \times [0,1] \rightarrow N$ is a smooth homotopy from $f_0$ and $f_1$, then the length of $F$ is given by
\begin{align*}
L (F) = \sup_{p \in M} \int_0^1 | \partial_s F (p, s) | ds.
\end{align*}
Then the distance $d (f_0, f_1)$ is the infimum of the lengths over all smooth homotopies from $f_0$ and $f_1$. The authors in \cite{schoen1979compact} deduced that the distance can be uniquely attained by a smooth homotopy $F$, called the geodesic homotopy, in which $s \mapsto F( p ; s)$ is a geodesic for all $p \in M$, when $(N, h)$ has nonpositive sectional curvature.

\begin{theorem}
Assume that $(M, HM, J_b, \theta)$ is a closed pseudo-Hermitian manifold and $(N, h)$ is a compact Riemannian manifold with nonpositive sectional curvature. If $f_0 (p, t)$ and $f_1 (p, t)$ are two solutions of the pseudo-harmonic flow with the homotopic initial data, then $t \mapsto d (f_0 (\cdot, t),  f_1 ( \cdot, t ))$ is nonincreasing.
\end{theorem}

\begin{proof}
Fixed $t_0 \geq 0$ and let $\phi (p; s)$ be the geodesic homotopy from $f_0 ( p, t_0)$ to $f_1 (p, t_0)$. Then we can obtain the solution $K $ of the pseudo-harmonic flow with the initial data $\phi(p; s)$ for any $s$. Thus $K (p, t - t_0 ;0)= f_0 (p, t)$ and $K (p, t - t_0 ;1)= f_1 (p, t)$ for $t \geq t_0$. By Lemma \ref{et1},
$|| \partial_s K (\cdot , t ; s)  ||_{C^0 (M)} $
is nonincreasing with respect to $t$ which yields that
\begin{align*}
d (f_0( \cdot, t ), f_1 (\cdot, t ) ) \leq& L \big( K ( \cdot, t- t_0; \cdot )  \big) = \sup_{p \in M} \int^1_0 \bigg| \frac{\partial K}{\partial s} (p , t- t_0; s) \bigg| d s \\
\leq & \int^1_0 \sup_{x \in M} \bigg| \frac{\partial K}{\partial s} (p , t- t_0; s) \bigg| d \tau \leq \int^1_0 \sup_{x \in M} \bigg| \frac{\partial K}{\partial s} (p , 0; s) \bigg| d \tau \\
= & d \big( f_0(\cdot, t_0 ), f_1 ( \cdot, t_0 ) \big).
\end{align*}
Hence we complete the proof.
\end{proof}

The pseudo-harmonic map $f_\infty$ of Theorem \ref{dt3} is also the solution of the pseudo-harmonic heat flow with the initial data $f_\infty$. So the distance of $f (p, t)$ and $f_\infty (p)$ is nonincreasing. Hence we have the following theorem.

\begin{theorem} \label{et2}
Assume that $(M, HM, J_b, \theta)$ is a closed pseudo-Hermitian manifold and $(N, h)$ is a compact Riemannian manifold with nonpositive sectional curvature. Let $\phi: M \rightarrow N$ be a smooth map. Then the solution $f (p, t)$ of the pseudo-harmonic flow
\begin{align*}
\frac{\partial f}{\partial t} =\tau (f), \mbox{ and }  f (p, 0) = \phi (p)
\end{align*}
exists, and it converges in $C^\infty (M,N)$ to a pseudo-harmonic map $f_\infty$ in the same homotopy class of $ \phi $.
\end{theorem}

This asymptotic behavior at infinity was first obtained by P. Hartman in \cite{hartman1967homotopic} in the harmonic case. Some other properties for harmonic maps given in \cite{hartman1967homotopic} can also be generalized to pseudo-harmonic maps.

\begin{theorem} \label{et5}
Assume that $(M, HM, J_b, \theta)$ is a closed pseudo-Hermitian manifold and $(N, h)$ is a compact Riemannian manifold with nonpositive sectional curvature. If $\Phi$ is the geodesic homotopy between two pseudo-harmonic maps $\phi_0 $ and $ \phi_1$, then $E_H (\Phi_s) = E_H (\phi_0) = E_H (\phi_1)$. Hence all pseudo-harmonic maps share the same horizontal energy.
\end{theorem}

\begin{proof}
Since $\Phi (p, \cdot)$ is geodesic for any $p$, we have
\begin{align*}
\frac{d}{ds} E_{H} (\Phi_s) = \int_M \la \nabla_{\partial_s} d_H \Phi, d_H \Phi  \ra = \int_M \la \nabla_{H} d \Phi (\partial_s), d_H \Phi  \ra,
\end{align*}
and
\begin{align*}
\frac{d^2}{ds^2} E_{H} (\Phi_s) =& \int_M \la \nabla_{H} d \Phi (\partial_s), \nabla_H d \Phi (\partial_s)  \ra + \int_M \la \nabla_{\partial_s} \nabla_{H} d \Phi (\partial_s), d_H \Phi  \ra \\
=& \int_M \big| \nabla_{H} d \Phi (\partial_s) \big|^2 + \int_M \la \nabla_{H} \nabla_{\partial_s} d \Phi (\partial_s), d_H \Phi  \ra \\
&+ \int_M \la R^N( d \Phi (\partial_s), d_H \Phi ) d \Phi (\partial_s), d_H \Phi \ra \\
\geq & 0.
\end{align*}
Thus $\dfrac{d}{ds} E_H (\Phi_s)$ is increasing. 
But because $\phi_0, \: \phi_1$ are pseudo-harmonic, 
$$ \frac{d}{ds} \bigg|_{s=0} E_{H} (\Phi_s) =   \frac{d}{ds} \bigg|_{s=1} E_{H} (\Phi_s) =0 $$ 
which implies $$\frac{d}{ds}  E_{H} (\Phi_s) \equiv 0. $$
Hence $ E_H (\Phi_s) = E_H (\phi_0) = E_H (\phi_1) $. 
\end{proof}

\begin{theorem} \label{et3}
Assume that $(M, HM, J_b, \theta)$ is a closed pseudo-Hermitian manifold and $(N, h)$ is a compact Riemannian manifold with nonpositive sectional curvature. If $\phi_0, \phi_1 : M \rightarrow N$ are homotopic pseudo-harmonic maps, then their geodesic homotopy $\Phi: M \times [0,1] \rightarrow N$ is pseudo-harmonic, i.e. $\Phi(p; s)$ is pseudo-harmonic for any $s\in [0,1]$. In particular, the set of pseudo-harmonic maps in the same homotopy class is connected.
\end{theorem}

\begin{proof}
The proof is by contradiction. Suppose that $\Phi_s$ is not pseudo-harmonic. Then there exists the pseudo-harmonic heat flow $f (p, t) $ with the initial data $\Phi_s$ and it will go to some pseudo-harmonic map $f_\infty$. By the horizontal energy estimate (Lemma \ref{dt1}),
$$ \frac{d}{dt} \bigg|_{t=0} E_H (f_t) = - \int_M \big| \tau  ( \Phi_s ) \big|^2 <0. $$
we have $ E_{H} (f_\infty) < E_{H} (\Phi_s) $. This leads a contradiction with Theorem \ref{et5} which implies that $E_H (f_\infty) = E_H (\phi_0) = E_H (\Phi_s)$.
\end{proof}

When the target manifold $(N, h)$ has negative sectional curvature, we could obtain more than the horizontal energy rigidity about pseudo-harmonic maps.

\begin{theorem}
Assume that $(M, HM, J_b, \theta)$ is a connected closed pseudo-Hermitian manifold and $(N, h)$ is a compact Riemannian manifold with negative sectional curvature. Let $\phi: M \rightarrow N$ be a pseudo-harmonic map. If its image is neither a single point nor a closed geodesic, then $\phi$ is the unique pseudo-harmonic map in its homotopy class.
\end{theorem}

\begin{proof}
By contradiction, we assume that $\tilde{\phi}$ is another pseudo-harmonic map homotopic to $\phi$ and $\phi$ is non-trivial. By Theorem \ref{et3}, their geodesic homotopy $\Phi (p ;s ) : M \times [0,1] \rightarrow N$ is pseudo-harmonic. Then a similar calculation as \eqref{e13} shows that
\begin{align*}
\frac{1}{2} \Delta_H |\partial_s \Phi|^2 =& \big| \nabla_H \partial_s \Phi \big|^2 - trace_{G_\theta} \la R^N ( d_H \Phi , \partial_s \Phi) \partial_s \Phi, d_H \Phi \ra \geq 0
\end{align*}
which yields that $|\partial_s \Phi| $ is nonzero constant by Bony's maximum principle (cf. \cite{bony1969principe,jerison1987subelliptic,jost1998subelliptic}). Thus $|\nabla_H \partial_s  \Phi | =0$ and
$$ d \Phi (\eta_\alpha) = a_\alpha (p) \partial_s \Phi $$
where $\{ \eta_\alpha \}$ is a local unitary frame of $T_{1,0} M$ and $a_\alpha$'s are some smooth functions on some open set of $M$. 
Hence the commutation relation \eqref{d13} implies that
\begin{align*}
d\Phi (T) = \frac{\ii}{2m} \big( a_{\ba \alpha} - a_{\alpha \ba} \big) \partial_s \Phi.
\end{align*}
Since $s \mapsto \Phi(p; s) $ is a geodesic, then $\nabla_T \partial_s \Phi = \nabla_{\partial_s} d \Phi (T) =0$. For any $X, Y \in TM$, we find
\begin{align*}
2 d \la d \Phi, \partial_s \Phi \ra (X, Y) =& X \la d \Phi (Y), \partial_s \Phi \ra- Y \la d \Phi (X), \partial_s \Phi \ra - \la d \Phi ([X, Y]), \partial_s \Phi \ra \\
=& \big\la \nabla_X (d \Phi (Y)) - \nabla_Y (d \Phi (X)) - d \Phi ([X, Y]), \partial_s \Phi \big\ra \\
= & 0
\end{align*}
due to the fact that the Levi-Civita connection of $(N, h)$ is torsion-free. For any $q \in M$, by Poincar\'e Lemma, there exists a neighborhood $U_q$ of $q$ and a smooth function $\psi$ defined on $U_q$ such that
\begin{align*}
\la d \Phi, \partial_s \Phi \ra = - d \psi \mbox{ and } \ \psi (q) =\frac{1}{2}.
\end{align*}
Hence the new function $ \Phi( p, s + \psi (p))$ is closed which implies that it is independent of $p \in U_q$. 
Let $\gamma_p $ be the geodesic satisfying $\gamma_p (s) = \Phi (p,s)$. For any $p \in U_q$ and sufficiently small $s$, we find 
\begin{align}
\gamma_p (s +\psi(p) ) = \Phi ( p, s+\psi(p) ) = \Phi (q, s + \psi(q)) = \gamma_{q} (s + \psi (q)) \label{e14}
\end{align} 
By the uniqueness of the geodesic, all geodesic $\gamma_p$ for $p \in U_q$ are just $\gamma_q$ and 
\begin{align*}
\phi (p) = \gamma_p (0) =\gamma_q ( \psi(q) - \psi(p) ). 
\end{align*}
Moreover, by a direct calculation, $ \Delta_H \psi =0 $ which implies that $\psi \equiv \frac{1}{2}$ or $\psi $ has no maximal or minimal point in $U_q$. In the former case, $ \phi (U_q) $ is a single point. In the latter case, by shrinking $U_q$, also denoted by $U_q$, $ \phi (U_q) $ will contain $\gamma_q \big( (\frac{1}{2}- \epsilon, \frac{1}{2} + \epsilon ) \big) $ for some $ \epsilon $ which implies that $ \phi (U_q) $ is an open set in $ \gamma_q $. 

By the connectedness and compactness of $M$, $\phi (M)$ lies in a geodesic $\gamma$ and is closed. Moreover, there exist two open sets $U$  and $V$ of $M$ such that $U \cap V \neq \emptyset $, $\phi (U)$ is an open set of $\gamma$ and $ \phi (V) $ is a finite set. Since $\phi$ is non-trivial, then $U \neq \emptyset$. Hence $ \phi (V) \subset \phi(U) $ which implies that $\phi (M)$ is an open set of $ \gamma $. This leads that $ \phi (M) $ is exactly $ \gamma $. Since $M$ is compact, so is $\gamma$. Thus $\gamma$ is closed which leads a contradiction with the assumption of the image of $\phi$.
\end{proof}

\begin{corollary}
Assume that $(M, HM, J_b, \theta)$ is a closed pseudo-Hermitian manifold and $(N, h)$ is a compact Riemannian manifold with negative sectional curvature. Let $\phi: M \rightarrow N$ be a pseudo-harmonic map. If $ rank(d \phi) \geq 2$ at some point, then $\phi$ is the unique pseudo-harmonic map in its homotopy class.
\end{corollary}

\section*{Acknowledgments}
The authors would like to express their thanks to Professor Yuxin Dong for constant encouragements and valuable discussions.

\bibliographystyle{plain}

\bibliography{renyang}

\vspace{12 pt}

Yibin Ren

\emph{College of Mathematics, Physics and Information Engineering}

\emph{Zhejiang Normal University}

\emph{Jinhua, 321004, Zhejiang, P.R. China}

allenryb@outlook.com

\vspace{12 pt}

Guilin Yang

\emph{School of Mathematics and Statistics}

\emph{Huazhong University of Science and Technology}

\emph{Wuhan 430074, P.R. China}

guilinyang.2016@gmail.com

\end{document}